\documentclass[11pt, reqno]{amsart}
\usepackage{amsmath,amsthm,amssymb,mathrsfs}
\usepackage[abbrev,non-sorted-cites]{amsrefs}
\usepackage{color}
\usepackage{verbatim}
\usepackage{datetime}
\usepackage{hyperref}
\usepackage{mathtools}
\usepackage{nicematrix}

\theoremstyle{plain}
\newtheorem{thm}{Theorem}[section]

\newtheorem{lemma}[thm]{Lemma}

\newtheorem{prop}[thm]{Proposition}
 
\theoremstyle{definition}
\newtheorem{defn}[thm]{Definition}

\newtheorem{rem}{Remark}

\newtheorem*{observation*}{Observation} 
\newtheorem*{ack*}{Acknowledgement}
\newtheorem*{ques*}{Question}

\theoremstyle{plain}

\newtheorem{known}{Theorem}

\numberwithin{equation}{section}

\newcommand\dd{{\mathrm d}}

\newcommand{\pa}{\partial}

\DeclareMathOperator{\specc}{Spec}

\begin{document}
		\title{FIRST EIGENVALUE CHARACTERIZATION OF CLIFFORD HYPERSURFACES AND VERONESE SURFACES}
	
	\author{Pei-Yi Wu}
	\address{Department of Mathematics\\
		Fudan University\\ Shanghai 200433\\ China}
	\email{pw17@fudan.edu.cn}
	
	\begin{abstract}
		We give an estimate for the first eigenvalue of the  Schr\"odinger operator $L:=-\Delta-\sigma$ which defined on the closed minimal submanifold $M^{n}$ in the unit sphere $\mathbb{S}^{n+m}$, where $\sigma$ is the square norm of the second fundamental form.
	\end{abstract}
	\maketitle
	\section{Introduction}\label{sc_intro}
	\hspace*{5mm}The study of rigidity theorems plays an important role in the theory of minimal submanifolds. Also there has been extensive research in the mathematics literature on rigidity theorems for minimal submanifolds in Spheres since the  pioneering results obtained by Simons (\cite{simons1968minimal}), Lawson (\cite{lawson1969local}) and Chern-Do Carmo-Kobayashi (\cite{chern1970minimal}). To be precise, denote $\sigma$ as the square norm of second fundamental form, 
	Let $M^n$ be a compact minimal submanifold in a unit sphere $\mathbb{S}^{n+m}$, they proved that if  $0\leq \sigma\leq \frac{n}{2-\frac{1}{m}}$, then either $\sigma=0$ or $\sigma=\frac{n}{2-\frac{1}{m}}$ and $M$ is the Clifford hypersurface or the Veronese surface in $\mathbb{S}^4$. 
	Later, Li~\cite{an1992intrinsic} and Xu-Chen~\cite{chen1993rigidity} improved the pinching number $\frac{n}{2-\frac{1}{m}}$ to $\frac{2n}{3}$ respectively, they proved that  if  $0\leq \sigma\leq \frac{2n}{3}$, then either $\sigma=0$ or $\sigma=\frac{2n}{3}$ and $M$ is the Veronese surface in $\mathbb{S}^4$. Recently, Lu generalized their result and got the following rigidity theorem, denote by $\lambda_2$  the second largest eigenvalue of the fundamental matrix(see Definition~\ref{funda}), 
	\begin{known}[Lu]~\cite{lu2011normal}\label{lu2011normal}  
		Let
		\[
		0\leq \sigma+\lambda_2\leq n.
		\]
		Then either $M$ is totally geodesic, or is one of the Clifford hypersurfaces  $M_{r,n-r}$ ($1\leq r\leq n$) in $\mathbb{S}^{n+m}, m\geq 1$, or is a Veronese surface in $\mathbb{S}^{2+m}, m \geq 2 $.
	\end{known}
	\begin{rem}\label{rmklu}
		Lu suggests that quantity $\sigma+\lambda_2$ might be the right object to study pinching theorems.
	\end{rem}
	In this paper, using Lu inequality ~\cite{lu2011normal}[Lemma 2] (see Lemma~\ref{lem2new}), we study the first eigenvalue of the Schr\"odinger operator $L:=-\Delta+q$, where $\,q $ is some continuous function on $M$. If there is a nonzero $f\in C^{\infty}(M)$ satisfying equation $Lf=\mu f$ , we call $\mu$  the eigenvalue of $L$. Since $\Delta$ is elliptic, so is $L$. And we have 
	\begin{align*}
		\specc(L)=\{\mu_{i};\mu_{1}<\mu_{2}\leq\mu_{3}\leq\cdots\}.
	\end{align*}
	we call $\mu_{1}$ the first eigenvalue of $L$.\\
	The pinching theorems mentioned above gives a characterization of Clifford hypersurfaces and Veronese Surfaces and their proof shares the same spirit of making good use of the Simons identity. Also mathematicians found that one can use similar argument to give an estimate of the first eigenvalue of the Schr\"odinger operator , which also can be seen as a new way to characterize the Clifford hypersurfaces and Veronese Surfaces. First, Simons~\cite{simons1968minimal} studied the Schr\"odinger operator $L_{I}:= -\Delta-\sigma$ of minimal hypersurfaces $M^{n} \to \mathbb{S}^{n+1} $, proved that its first eigenvalue $\mu_{1}^{I}\leq -n$ if $M$ is not totally geodesic. Later, Wu~\cite{wu1993new} and Perdomo~\cite{perdomo2002first} proved that if $\mu_{1}^{I}\geq -n$, then $M$ is either totally geodesic or a Clifford hypersurfaces independently. And Wu~\cite{wu1993new} also got the following results: Define $L_{II}:= -\Delta-(2-\frac{1}{m})\sigma$ of minimal submanifold $M^{n} \to \mathbb{S}^{n+m} $ and $L_{III}:= -\Delta-\frac{3}{2}\sigma$ of minimal submanifold $M^{n} \to \mathbb{S}^{n+m}, m\geq 2 $, denote $\mu_{1}^{II}$  and $\mu_{1}^{III}$ by their first eigenvalue respectively. Then, for $L_{II}$, $\mu_{1}^{II}\leq -n$ if $M$ is not totally geodesic, and if $\mu_{1}^{II}\geq -n$, then $M$ is either totally geodesic or $\mu_{1}^{II}=-n$ and $M$ is either the Clifford hypersurfaces or the Veronese surfaces. For $L_{III}$, $\mu_{1}^{III}\leq -n$ if $M$ is not totally geodesic, and if $\mu_{1}^{III}\geq -n$, then $M$ is either totally geodesic or $\mu_{1}^{III}=-n$ and $M$ is the Veronese surfaces. Similar result also happens in Legendrian case, by using a pinching rigidity result in ~\cite{luo2022optimal}, Yin-Qi~\cite{yin2022sharp} got a sharp estimate for the first eigenvalue of Schr\"odinger operator defined on minimal Legenderian submanifold $M^{3} \to \mathbb{S}^{7} $.\\  
	
	Inspired by the correspondence between the pinching theorems and the estimates of the first eigenvalue of certain  Schr\"odinger operators , one expect that  these phenomenon would happen to Lu's rigidity theorem~\ref{lu2011normal} and that leads to our main theorem.
	Define the Schr\"odinger operator as $$L:=-\Delta-\sigma$$ and denoted the first eigenvalue of $L$ by $\mu_{1}$, our main theorem  is 
	\begin{thm}[Main Theorem]~\label{main}
		Let $M^n$ be a closed minimal submanifold in $\mathbb{S}^{n+m}(1)$, then 
		\begin{align}
			\mu_{1} \leq -n + \max_{p\in M}\lambda_{2}
		\end{align}
		if $M$ is not totally geodesic. Moreover, if $\mu_{1} \geq -n+\max_{p\in M}\lambda_{2}$, then either $\mu_{1}=0$ and $M$ is totally geodesic or  $\mu_{1}=-n+\max_{p\in M}\lambda_{2}$ and $M$ is the Clifford hypersurfaces in $\mathbb{S}^{n+m}(1)$ or is the Veronese surface in $\mathbb{S}^{2+m}(1)$.
	\end{thm}
	\section{Preliminaries and Lu's inequality }\label{sect:2}~
	Let $M^n$ be a compact minimal submanifold in a unit sphere $\mathbb{S}^{n+m}$.
	We shall make use of the following convention on the range of indices:
	$$ 1\leq A,B,C,\ldots\leq n+m;\ 1\leq i,j,k,\ldots\leq n;\ n+1\leq
	\alpha,\beta,\gamma,\ldots\leq n+m.$$ 
	We choose a local field of orthonormal frames $\{e_{1},e_{2},\dots,e_{n+m}\}$ in $\mathbb{S}^{n+m}$ such that  when restricting on $M$, the $\{e_{1},e_{2},\dots,e_{n}\}$ are tangent to $M$ and  $\{e_{n+1},e_{n+2},\dots,e_{n+m}\}$ are normal to $M$. Also $\{\omega_{1},\dots,\omega_{n+m}\}$ is the corresponding dual frame.
	It is well know that 
	\begin{eqnarray}&&\omega_{\alpha i}=\sum_{j} h^{\alpha}_{ij}\omega_{j},  \, \,
		h^{\alpha}_{ij}=h^{\alpha}_{ji},\nonumber\\
		&&h=\sum_{\alpha,i,j}h^{\alpha}_{ij}\omega_{i}\otimes\omega_{j}\otimes
		e_{\alpha},\,\,H=\frac{1}{n}\sum_{\alpha,i}h^{\alpha}_{ii}e_{\alpha},\nonumber
		\\
		&&R_{ijkl}=\delta_{ik}\delta_{jl}-\delta_{il}\delta_{jk}+\sum_{\alpha}(h^{\alpha}_{ik}h^{\alpha}_{jl}-h^{\alpha}_{il}h^{\alpha}_{jk}), \\ \label{gauss}
		&&R_{\alpha\beta kl}=\sum_{i}(h^{\alpha}_{ik}h^{\beta}_{il}-h^{\alpha}_{il}h^{\beta}_{ik}),\\ \label{ricci}
		&&h^{\alpha}_{ijk}=h^{\alpha}_{ikj},\label{codazzi}
	\end{eqnarray}
	where $h, H, R_{ijkl}, R_{\alpha\beta kl},$ 
	are the second fundamental form, the mean
	curvature vector, the curvature tensor, and the normal curvature tensor
	of $M$ respectively. We define
	$$\sigma=|h|^{2}, \ A_{\alpha}=(h^{\alpha}_{ij})_{n\times
		n}.$$
	Denote by $h_{ijk}^\alpha$  the component of the covariant derivative of $h_{ij}^\alpha$, defined by
	\begin{equation}\label{okj}
		h_{ijk}^\alpha\omega_k=dh_{ij}^\alpha-\sum_{l}h_{il}^\alpha\omega_{lj}-\sum_{l}h_{lj}^\alpha\omega_{li}+\sum_{\beta}h_{ij}^\beta\omega_{\alpha\beta}.
	\end{equation}
	From Gauss-Codazzi-Ricci equation above, the following Simons identity is well known, 
	\begin{align}	
		\sum_{i,j,}h_{ij}^{\alpha}\Delta h_{ij}^{\alpha}=n\,|A_{\alpha}|^{2} +\sum_{t}{\rm Tr}(A_{\alpha}A_{\beta}-A_{\beta}A_{\alpha})^2-\sum_{t}({\rm Tr}\,A_{\alpha}A_{\beta})^2 \label{simons}
	\end{align}
	Now, we  introduce the  Lu inequality~\cite{lu2011normal}[Lemma 2](see Lemma~\ref{lem2new}),  which is the main tool in the proof of  Theorem~\ref{lu2011normal}, and the Theorem~\ref{main}. And the proof of the  Lu inequality relies on a algebraic inequality ~\cite{lu2011normal}[Lemma 1], which we find that there are more cases when the equality holds and use Lagrange Multiplier method to give another proof. 
	To be precise, we restate Lu's lemma ~\cite{lu2011normal}[Lemma 1] as follows
	\begin{lemma} \label{lem1new}
		Suppose $\eta_1,\cdots,\eta_n$ are real numbers and 
		\[
		\eta_1+\cdots+\eta_n=0,\quad
		\eta_1^2+\cdots+\eta_n^2=1.
		\]
		Let $r_{ij}\geq 0$ be nonnegative numbers for $i<j$. Then we have
		\begin{equation}\label{use} 
			\sum_{i<j}(\eta_i-\eta_j)^2 r_{ij}\leq\sum_{i<j}r_{ij}+{\rm Max} (r_{ij}).
		\end{equation}
		If $\eta_1\geq\cdots\geq\eta_n$, and $r_{ij}$ are not simultaneously  zero, then the equality in~\eqref{use} holds in one of the following cases: Fix a integer $k, k \in \{1,\dots,n-1\}$, 
		\begin{enumerate}
			\item $r_{ij}=0$ if $2\leq i<j$, $r_{12}=\dots=r_{1k}=0, r_{1\,k+1}=\dots=r_{1n}>0$ and 
			$$\eta_{1}=\frac{\sqrt{n-k}}{\sqrt{n-k+1}},\qquad\eta_{2}=\dots=\eta_{k}=0,\qquad \eta_{k+1}=\dots=\eta_{n}=\frac{-1}{\sqrt{(n-k+1)(n-k)}}.$$
			\item $r_{ij}=0$ if $i<j<n$, $r_{n-1\,n}=\dots=r_{n-k+1\,n}=0, r_{n-k\,n}=\dots=r_{1n}>0$ 
			$$\eta_{n}=\frac{-\sqrt{n-k}}{\sqrt{n-k+1}},\qquad\eta_{n-1}=\eta_{n-2}\dots=\eta_{n-k+1}=0,\qquad \eta_{n-k}=\dots=\eta_{1}=\frac{1}{\sqrt{(n-k+1)(n-k)}}.$$
		\end{enumerate}
	\end{lemma}
	\begin{rem}
		We prove this by two steps. Our first step is the same as Lu's original proof ~\cite{lu2011normal}[Lemma 1] which reduces the problem  to proving the following inequality 
		$$\sum_{1<j}(\eta_1-\eta_j)^2 r_{1j}\leq\sum_{1<j} r_{1j}+\underset{1<j}{\rm Max}\, (r_{1j}).$$
		Then, we apply Lagrange Multiplier method to prove this inequality, and for reader's convenience, we will also include the first step in our proof.
	\end{rem}
	\begin{proof}
		First step. Assume
		$
		\eta_1\geq\cdots\geq\eta_n.
		$
		If $\eta_1-\eta_n\leq 1$ or $n=2$, then ~\eqref{use} is trivial. So we assume $n>2$,  and
		\[
		\eta_1-\eta_n>1.
		\]
		Observing that $\eta_i-\eta_j< 1$ for $2\leq i<j\leq n-1$. Otherwise, one could have
		\[
		1\geq\eta_1^2+\eta_n^2+\eta_i^2+\eta_j^2\geq\frac 12
		((\eta_1-\eta_n)^2+(\eta_i-\eta_j)^2)>1,
		\]
		which is a contradiction. 
		
		Using the same reasoning,
		if  $\eta_1-\eta_{n-1}>1$,  then we have $\eta_2-\eta_{n}\leq 1$; and if $\eta_2-\eta_{n}>1$, then we have $\eta_1-\eta_{n-1}\leq 1$.
		Replacing $\eta_1,\cdots,\eta_n$ by $-\eta_n,\cdots,-\eta_1$ if necessary, we can always assume that $\eta_2-\eta_{n}\leq 1$.
		Thus $\eta_i-\eta_j\leq 1$ if $2\leq i<j$, and ~\eqref{use} is implied by the following inequality
		\begin{equation}\label{use2}
			\sum_{1<j}(\eta_1-\eta_j)^2 r_{1j}\leq\sum_{1<j} r_{1j}+\underset{1<j}{\rm Max}\, (r_{1j}).
		\end{equation}
		Before proving the inequality ~\eqref{use2}, we can observe that if the equality in ~\eqref{use} holds, we must have $\eta_1-\eta_n>1$, otherwise,
		\[
		\sum_{i<j}(\eta_i-\eta_j)^2 r_{ij}\leq\sum_{i<j}r_{ij}<\sum_{i<j}r_{ij}+{\rm Max} (r_{ij}),
		\]
		which is a contradiction.
		
		Notice that, when $\eta_1-\eta_n>1$, by the discussion above, we have $\eta_i-\eta_j<1$, for $2\leq i<j<n$. So, we have $r_{ij}=0$, for $2\leq  i<j<n$, otherwise, by ~\eqref{use2} we know that the equality in ~\eqref{use} will not hold. Thus, when discussing how the equality holds in ~\eqref{use}, we only need to analyse the situation of inequality  ~\eqref{use2}(Again we should remind the reader that above analysis comes from Lu's original proof ~\cite{lu2011normal}[Lemma 1]) \\
		
		Second step. Let $s_j=r_{1j}$, where $j=2,\cdots,n$, then we write ~\eqref{use2} as follow
		\begin{equation}\label{use3}
			\sum_{1<j}(\eta_1-\eta_j)^2 s_{j}\leq\sum_{1<j} s_{j}+\underset{1<j}{\rm Max}\, (s_{j}).
		\end{equation}
		Denote 
		$$f(\eta_1,\eta_2,\dots,\eta_{n-1},\eta_{n})=	\sum_{1<j}(\eta_1-\eta_j)^2 s_{j}$$
		we apply Lagrange multiplier method to $f$ with constraints
		\begin{align}
			\begin{split}		
				\eta_1+\cdots+\eta_n=0,\quad
				\eta_1^2+\cdots+\eta_n^2-1=0.\label{constrains}
			\end{split}
		\end{align}
		Consider the function
		$$\Phi(\eta_1,\eta_2,\dots,\eta_{n-1},\eta_{n})=	\sum_{1<j}(\eta_1-\eta_j)^2 s_{j}+\lambda\,(\eta_1+\cdots+\eta_n)+
		\mu\,(\eta_1^2+\cdots+\eta_n^2-1)
		$$
		where $\lambda$ and $\mu$ are the Lagrange multipliers.\\
		
		Partial differentiating with respect to each variable, and let them equal to zero, we obtain the following equations
		\begin{align}
			\frac{\pa \Phi}{\pa \eta_1}&=	\sum_{1<j}2(\eta_1-\eta_j) s_{j}+\lambda+2\mu\,\eta_1=0,\label{lmm1}\\\label{lmm2}
			\frac{\pa \Phi}{\pa \eta_j}&=	-2(\eta_1-\eta_j) s_{j}+\lambda+2\mu\,\eta_{j}=0,\qquad\text{where $j=2,\dots,n-1,n$}.
		\end{align}
		by
		$$\sum_{i=1}^n 	\frac{\pa \Phi}{\pa \eta_i}=n\lambda=0$$
		$$\sum_{i=1}^n 	\eta_{i}\frac{\pa \Phi}{\pa \eta_i}=	2\sum_{1<j}(\eta_1-\eta_j)^2 s_{j}+2\mu=0$$
		we can see
		$$\lambda=0$$
		$$\sum_{1<j}(\eta_1-\eta_j)^2 s_{j}=-\mu$$
		Hence the critical values of $f$ are given by $-\mu$. Now, we assume $-\mu\neq0$.
		Meanwhile, we can also assume that  $\mu+\underset{1<j}{\rm Max}\, (s_{j})< 0$, otherwise, we have
		$$ -\mu=\sum_{1<j}(\eta_1-\eta_j)^2 s_{j}\leq\underset{1<j}{\rm Max}\, (s_{j})<\sum_{1<j} s_{j}+\underset{1<j}{\rm Max}\, (s_{j})$$
		Then by  ~\eqref{lmm2}, we have
		\begin{align}
			\eta_{j}=\frac{\eta_{1}\,s_{j}}{\mu+s_j},\qquad j=2,\dots,n-1,n.\label{etaj}
		\end{align}
		Substitute ~\eqref{etaj} into $\eta_1+\cdots+\eta_n=0$, we have 
		$$1+\sum_{1<j}\frac{\,s_{j}}{\mu+s_j}=0$$
		Hence 
		\begin{align}
			0=1+\sum_{1<j}\frac{\,s_{j}}{\mu+s_j}\geq1+\sum_{1<j}\frac{\,s_{j}}{\mu+\underset{1<j}{\rm Max}\, (s_{j})}.\label{coreineq0}
		\end{align}
		Multiply both sides of  ~\eqref{coreineq0} by $\mu+\underset{1<j}{\rm Max}\, (s_{j})$, we get 
		\begin{align}
			-\mu=\sum_{1<j}(\eta_1-\eta_j)^2 s_{j}\leq\sum_{1<j} s_{j}+\underset{1<j}{\rm Max}\, (s_{j}). \label{coreineq}
		\end{align}
		Notice that for any $j$, if $s_{j}>0$, we have $0>\frac{\,s_{j}}{\mu+s_j}\geq\frac{\,s_{j}}{\mu+\underset{1<j}{\rm Max}\, (s_{j})}$,
		And the equality in ~\eqref{coreineq} holds is equivalent to the equality in ~\eqref{coreineq0} holds. If the equality in ~\eqref{coreineq0} holds, we have,  for each $j>1$,
		$$\frac{\,s_{j}}{\mu+s_j}=\frac{\,s_{j}}{\mu+\underset{1<i}{\rm Max}\, (s_{i})}$$
		which means that either $s_{j}=0$ or the  nonzero $s_{j}=\underset{1<i}{\rm Max}\, (s_{i})$, that means all nonzero $s_{j}$ are equal.\\
		Thus, by ~\eqref{etaj} and the assumpsion above, fix a integer $k, k \in \{1,\dots,n-1\}$, we have $n-1$ cases, for each $k$, 
		$$\eta_{1}=\frac{\sqrt{n-k}}{\sqrt{n-k+1}},\qquad\eta_{2}=\dots=\eta_{k}=0,\qquad \eta_{k+1}=\dots=\eta_{n}=\frac{-1}{\sqrt{(n-k+1)(n-k)}}.$$
		Case(ii) is just a permutation of Case(i) under different assumption at the begining, so we finish the proof.
	\end{proof}
	\begin{rem}
		\begin{enumerate}
			\item When $k$ takes values $n-1$ or $1$ in Case(1) of Lemma~\ref{lem1new}, it corresponds to the Case(1) and Case(2) in  ~\cite{lu2011normal}[Lemma 1] respectively.
			\item This new version would change Lemma 2 in ~\cite{lu2011normal}, but Lu's rigidity theorem would still hold true, we will discuss this later.
		\end{enumerate}
	\end{rem}
	Define the inner product of two $n\times n$ matrices $A,B$ as $\langle A,B\rangle={\rm Tr}\,  AB^{\top}$ and let $||A||^2=\langle A,A\rangle=\sum_{i,j}a_{ij}^{2}$, where $(a_{ij})$ are the entries of $A$. So the Lu inequality~\cite{lu2011normal}[Lemma 2] becomes
	\begin{lemma}\label{lem2new}
		Let $A_{1}$ be an $n\times n$  diagonal  matrix of norm $1$. Let $A_2,\cdots, A_m$ be symmetric matrices such that
		\begin{enumerate}
			\item $\langle A_\alpha,A_\beta\rangle=0$ if  $\alpha\neq \beta$;
			\item $||A_2||\geq\cdots\geq||A_m||$.
		\end{enumerate}
		Then we have
		\begin{equation}\label{luineq}
			\sum_{\alpha=2}^m||[A_1,A_\alpha]||^2\leq \sum_{\alpha=2}^m ||A_\alpha||^2+||A_2||^2.
		\end{equation}
		The  equality in~\eqref{lem2new} holds if and only if, after an orthonormal base change and up to a sign, fix a integer $k, k \in \{1,\dots,n-1\}$, for each $k$, 
		\begin{equation}
			A_1=\begin{pmatrix}
				\frac{\sqrt{k}}{\sqrt{k+1}} &0\\0&-\frac{1}{\sqrt{k(k+1)}}\\
				&&-\frac{1}{\sqrt{k(k+1)}}\\
				&&&\ddots\\
				&&&&-\frac{1}{\sqrt{k(k+1)}}\\
				&&&&&0\\
				&&&&&&\dots\\
				&&&&&&&0
			\end{pmatrix},
		\end{equation}
		$A_{i}$ is $\mu$ times the matrix whose only nonzero entries are $1$ at the $(1,i)$ and $(i,1)$ places, where $i=2,\cdots,k+1$. And
		$A_{k+2}=\cdots =A_m=0$.
	\end{lemma}
	Now we briefly review the proof of Lu' rigidity theorem to set up the notations and get some formulae for later use. We first defined the fundamental matrix as follows
	\begin{defn}\label{funda} The fundamental matrix $S$ of $M$ is an $m\times m$ matrix-valued function defined as $S=(a_{{\alpha}{\beta}})$, where
		\[
		a_{{\alpha}{\beta}}=\langle A_{\alpha}, A_{\beta}\rangle.
		\]
		We let $\lambda_1\geq\cdots\geq\lambda_m$ be the set of eigenvalues of the fundamental matrix $S$. In particular, $\lambda_1$ is the largest eigenvalue and $\lambda_2$ is the second largest eigenvalue of the matrix $S$.
	\end{defn}
	Using the above notation, $\sigma$ is the trace of the fundamental matrix: $\sigma=\lambda_1+\cdots+\lambda_m$.
	Let $p\geq 2$ be a positive integer, define 
	\[
	f_p:={\rm Tr}\,(S^p)=\sum_{{\alpha_1},\dots,{\alpha_p}} a_{{\alpha_1 }{\alpha_2 }} \, a_{{\alpha_2 }{\alpha_3 }}\dots \, a_{{\alpha_p }{\alpha_1 }} .
	\]
	We assume that at $x$,
	\[
	\lambda_1=\cdots=\lambda_r>\lambda_{r+1}\geq\cdots\geq\lambda_m.
	\]
	and set $g_p:=(f_p)^{\frac{1}{p}}$, using the Simons identity ~\eqref{simons} and Lemma \ref{lem2new}, Lu derived the following inequality
	\begin{prop}[Lu~\cite{lu2011normal}]
		Under notation above, we have 
		\begin{align}
			|\nabla f_p|^2&\leq p^2 f_p \sum_{k,\alpha}\lambda_\alpha^{p-2}\left(\nabla_{\frac{\pa}{\pa x_k}} a_{\alpha\alpha}\right)^2.\label{spherelemmause1}\\
			\begin{split}
				\Delta g_p&=\frac 1p f_p^{\frac 1p  -1}\Delta f_p+\frac 1p (\frac 1p-1) f_p^{\frac 1p -2}|\nabla f_p|^2\\
				&\geq  2f_p^{\frac 1p-1}\sum_\alpha\left(\lambda_\alpha^{p-1}\sum_{i,j,k}(h_{ijk}^\alpha)^2\right)\\&
				\quad+2f_p^{\frac 1p-1}( r||A_{1}||^{2p}(n-||A_{1}||^2-\sum_{\alpha=2}^m||A_{\alpha}||^2-\lambda_2)-3mn\lambda_{r+1}^{p}).\label{spherelemmause2}
			\end{split}
		\end{align}
	\end{prop}
	Then integrating both sides of ~\eqref{spherelemmause2}, letting $p \to \infty$, since $\lambda_{r+1}^p/f_p\to 0$ as $p$ tends to $\infty$, Lu derived 
	\begin{align}
		\int_M\sum_{i,j,k}\sum_{\alpha\leq r}(h^\alpha_{ijk})^2+||A_1||^2(n-||A_1||^2-\sum_{\alpha=2}^m||A_\alpha||^2-\lambda_2)
		\leq 0.\label{spherelemmause3}
	\end{align}
	When  the equality holds in  ~\eqref{spherelemmause3}, the equality in~\eqref{luineq} also holds , $A_{\alpha}$ takes the form in Lemma ~\ref{lem2new}. Then, using the structure equation case by case, Lu proved Theorem~\ref{lu2011normal}.
	
	\begin{rem}
		Although there are more cases when the  equality in~\eqref{luineq}  holds, using similar argument as Lu did in the original proof, we can rule out the new case. To be precise, for each $k$, let $n>k+1, j\geq k+2$, then from $0=\dd h_{1j}^{n+1}=h_{11}^{n+1}\omega_{1j}$ ,we conclude  $\omega_{1j}=0$ if $j\geq k+2$. Similarly, by computing $\dd h_{ij}^{n+1}, i\in{2\dots k+1}$, we also have $\omega_{2j}=\dots=\omega_{k+1\,j}=0$ for $j\geq k+2$. Thus by the structure equations, we have
		\[
		0=d\omega_{1j}=\omega_1\wedge\omega_j,
		\]
	\end{rem} 
	a contradiction if $n>k+1$, thus Theorem ~\ref{lu2011normal} still holds true.
	
	\section{Proof of the Main Theorem}

	Let $g_{\epsilon}=(g_{p}+\epsilon)^{\frac{1}{2}}$, where $\epsilon>0$ is a constant.
	We first prove the inequality in the main theorem, that is
	\begin{prop}\label{propinequal}
		Let $M^n$ be a closed nontotally geodesic minimal submanifold in $\mathbb{S}^{n+m}(1)$, then 
		\begin{align}
			\mu_{1} \leq -n + \max_{p\in M}\lambda_{2} -\frac{2}{n+2}\frac{\int_{M}
				\left[\frac{1}{r}\,\sum\limits_{i,j,k}\sum\limits_{{\alpha}\leq n+r}(h^{\alpha}_{ijk})^2\right]}{\int_{M}\lambda_{1}}
		\end{align}
	\end{prop}
	\begin{proof}
		By direct computation, using ~\eqref{spherelemmause2}, 
		\begin{align}
			\begin{split}
				\Delta g_{\epsilon}&= \frac{1}{2}\,(g_{p}+\epsilon)^{-\frac{1}{2}}\,\Delta g_{p}-\frac{1}{4}\,|\nabla g_{p}|^{2}\,(g_{p}+\epsilon)^{-\frac{3}{2}}\\
				&=\frac{1}{2}\,(g_{p}+\epsilon)^{-\frac{1}{2}}\,\left(\frac 1p f_p^{\frac 1p  -1}\Delta f_p+\frac 1p (\frac 1p-1) f_p^{\frac 1p -2}|\nabla f_p|^2\right)-\frac{1}{4}\,|\nabla g_{p}|^{2}\,(g_{p}+\epsilon)^{-\frac{3}{2}}\\
				&\geq\frac{1}{2}\,(g_{p}+\epsilon)^{-\frac{1}{2}}\,
				\Bigg( \Bigg.
				2f_p^{\frac 1p-1}\sum_{\alpha}\left(\lambda_{\alpha}^{p-1}\sum_{i,j,k}(h_{ijk}^{\alpha})^2\right)+2f_p^{\frac 1p-1}( r||A_{1}||^{2p}(n+1-||A_{1}||^2-\sum_{{\alpha}=2}^m||A_{\alpha}||^2-\lambda_2))\\
				&-6nm \,f_p^{\frac 1p-1}\lambda_{r+1}^{p}
				\Bigg. \Bigg)
				-\frac{1}{4}\,|\nabla g_{p}|^{2}\,(g_{p}+\epsilon)^{-\frac{3}{2}}\\
				&\geq\underbrace{(g_{p}+\epsilon)^{-\frac{3}{2}}\left[(g_{p}+\epsilon)\,f_p^{\frac 1p-1}\sum_{\alpha}\left(\lambda_{\alpha}^{p-1}\sum_{i,j,k}(h_{ijk}^{\alpha})^2\right)-\frac{1}{4}\,|\nabla g_{p}|^{2}
					\right]}_{I}\\
				&+\underbrace{(g_{p}+\epsilon)^{-\frac{1}{2}}\left[f_p^{\frac 1p-1}( r||A_{1}||^{2p}(n-||A_{1}||^2-\sum_{{\alpha}=2}^m||A_{\alpha}||^2-\lambda_2))
					-3nm\, f_p^{\frac 1p-1}\lambda_{r+1}^{p}\right]}_{II}.
			\end{split}
		\end{align}
		To deal with $I$, follows from $(1.9)$ and $(1.11)$ in ~\cite{shen1988curvature}[Proposition 1], we have  following Lemma
		
		\begin{lemma}[Shen~\cite{shen1988curvature}]~\label{lem:shensphere2}
			Let $M^n$ be a closed minimal submanifold in $\mathbb{S}^{n+m}(1)$, then we have 
			\begin{align}
				|\nabla(|A_{\alpha}|^{2})|^{2} \leq \frac{4n}{n+2}|A_{\alpha}|^{2}\, \left[\sum_{i,j,k}(h_{ijk}^{\alpha})^{2}\right].		
			\end{align}
		\end{lemma}
		Then, applying Lemma~\ref{lem:shensphere2} to ~\eqref{spherelemmause1}, we have
		\[
		|\nabla g_{p}|^{2}=\frac{1}{p^{2}}f_p^{\frac{2}{p}-2}|\nabla f_{p}|^{2}\leq f_p^{\frac{2}{p}-1}\sum_{\alpha}\lambda_{\alpha}^{p-2}|\nabla \lambda_{\alpha}|^{2}\\
		\leq f_p^{\frac{2}{p}-1}\sum_{\alpha}\,\frac{4n}{n+2}\left(\lambda_{\alpha}^{p-1}\,\sum_{i,j,k}(h_{ijk}^{\alpha})^2\right)
		\]
		Thus, for $I$, we have 
		\begin{align*}
			I \geq &(g_{p}+\epsilon)^{-\frac{3}{2}}\,\left[(g_{p}+\epsilon)\,f_p^{\frac 1p-1}\sum_{\alpha}\left(\lambda_{\alpha}^{p-1}\,\sum_{i,j,k}(h_{ijk}^{\alpha})^2\right)-\frac{1}{4}\,\frac{1}{p^{2}}f_p^{\frac{2}{p}-2}|\nabla f_{p}|^{2}
			\right]\\
			\geq&(g_{p}+\epsilon)^{-\frac{3}{2}}\,\left[(g_{p}+\epsilon)\,f_p^{\frac 1p-1}\sum_{\alpha}\left(\lambda_{\alpha}^{p-1}\,\sum_{i,j,k}(h_{ijk}^{\alpha})^2\right)-\frac{1}{4}\,f_p^{\frac{2}{p}-1}\,\frac{4n}{n+2}\sum_{\alpha}\left(\lambda_{\alpha}^{p-1}\,\sum_{i,j,k}(h_{ijk}^{\alpha})^2\right)\,\right]\\
			\geq&\frac{2}{n+2}(g_{p}+\epsilon)^{-\frac{1}{2}}\,f_p^{\frac{1}{p}-1}
			\left[\sum_{\alpha}\left(\lambda_{\alpha}^{p-1}\,\sum_{i,j,k}(h_{ijk}^{\alpha})^2\right)\right]\\
			\geq &0
		\end{align*}%
		Meanwhile, by definition, put $g_{\epsilon}$ into $\mu_{1}=\inf_{f\in C^{\infty}(M)}\, \frac{\int_{M}L(f)f}{\int_{M}f^{2}}$, we get 
		\begin{align*}
			\begin{split}
				\mu_{1} \leq&	\frac{\int_{M}L(g_{\epsilon})g_{\epsilon}}{\int_{M}g_{\epsilon}^{2}}=\frac{\int_{M}-g_{\epsilon}\Delta \,g_{\epsilon} -\sigma\,g_{\epsilon}^{2}}{\int_{M}g_{\epsilon}^{2}}\\
				=&\frac{\int_{M}-g_{\epsilon}(I+II) -\sigma\,g_{\epsilon}^{2}}{\int_{M}g_{\epsilon}^{2}}\\
				\leq &\frac{\int_{M}-\frac{2}{n+2}f_p^{\frac{1}{p}-1}
					\left[\sum\limits_{\alpha}\left(\lambda_{\alpha}^{p-1}\,\sum\limits_{i,j,k}(h_{ijk}^{\alpha})^2\right)\right]}{\int_{M}g_{\epsilon}^{2}}\\
				&+\frac{\int_{M}-\left[f_p^{\frac 1p-1}( r||A_{1}||^{2p}(n-||A_{1}||^2-\sum\limits_{{\alpha}=2}^{m}||A_{\alpha}||^2-\lambda_2))
					-3nm f_p^{\frac 1p-1}\lambda_{r+1}^{p}\right]}{\int_{M}g_{\epsilon}^{2}}\\
				&-\frac{\int_{M}\sigma\,g_{\epsilon}^{2}}{\int_{M}g_{\epsilon}^{2}}.
			\end{split}
		\end{align*}
		Then, leting $p \to \infty$ and   $\epsilon \to 0$, using the fact that $\lambda_{r+1}^p/f_p\to 0$ almost everywhere when $p\to \infty$ , we finish the proof.
	\end{proof}
	
	Now, we prove the main theorem,
	\begin{proof}
		From the proof of Proposition~\ref{propinequal}, we know that if $$ \mu_{1}\geq-n+ \max_{p\in M}\lambda_{2},$$
		then either $M$ is totally geodesic, so $\mu_{1}=0$, or $ \mu_{1}=-n+ \max\limits_{p\in M}\lambda_{2}$, we have $$\frac{1}{r}\,\sum\limits_{i,j,k}\sum_{{\alpha}\leq n+r}(h^{\alpha}_{ijk})^2=0,$$ we claim that $\sigma$ is a constant.
		By Lemma~\ref{lem2new}, there are following cases:
		\begin{enumerate}
			\item[Case1] \quad If $A_{1} \neq 0$, and $A_{2}=A_{3}=\dots=A_{m}=0$, so $\sigma=||A_{1}||^{2}=\lambda_{1}$, by Lemma~\ref{lem:shensphere2}, we know $\sigma$ is a constant.
			\item[Case2]\quad Fix a positive integer $1\leq k\leq n-1.$ For each $k$, if $$A_{k+2}=\dots=A_{m}=0$$ and
			\begin{equation}
				A_1=\lambda\begin{pmatrix}
					\frac{\sqrt{k}}{\sqrt{k+1}} &0\\0&-\frac{1}{\sqrt{k(k+1)}}\\
					&&-\frac{1}{\sqrt{k(k+1)}}\\
					&&&\ddots\\
					&&&&-\frac{1}{\sqrt{k(k+1)}}\\
					&&&&&0\\
					&&&&&&\dots\\
					&&&&&&&0
				\end{pmatrix},
			\end{equation}
			\quad
			\begin{align}\begin{split}
					A_{i}=\quad
					\begin{pNiceMatrix}[first-col,first-row]
						&	 &       & (1,i)  &      &   \\
						&	0&\ldots &\frac{\mu}{\sqrt{k(k+1)}} & \ldots  &0\\
						&	\vdots   &0  &0          & \ldots  &0 \\
						(i,1)\quad\,	&	\frac{\mu}{\sqrt{k(k+1)}}  &0   &0  & \ldots  &0 \\
						&	\vdots   &\vdots     &\vdots     & \ddots  &\vdots \\
						&	0    &0          &0          & \ldots  &0 \\
						&
					\end{pNiceMatrix}
				\end{split} 
			\end{align}
			where $i=2,\dots,k+1.$ 
			Since $\sum\limits_{i,j,k}(h^{n+1}_{ijk})^2=0$ by ~\eqref{okj}, we know $\lambda$ is a constant. And from $\lambda_{2}=\max\limits_{p\in M}\lambda_{2}$, $\mu$ is a constant too, thus $\sigma$ is a constant.
		\end{enumerate}
		Since $\sigma$ is a constant when $\mu_{1}=-n+ \max\limits_{p\in M}\lambda_{2}$, so the first eigenvalue of $L$ is $-\sigma$, 
		that implies $\sigma+\lambda_{2}=n$. Then, by Theorem ~\ref{lu2011normal}, $M$ is either one of the Clifford hypersurfaces or the Veronese Surface.
	\end{proof}
\begin{ack*}
	The authors would like to thank his advisor Prof. Ling Yang for many useful suggestions.
\end{ack*}
\bibliographystyle{plain}

\end{document}